\documentclass[12pt,reqno]{amsart}
\usepackage{amsmath}
\usepackage{amssymb}
\usepackage{amsthm}
\usepackage{bm}
 \usepackage{url}
\numberwithin{equation}{section}
\newtheorem{theorem}{Theorem}[section]
\newtheorem{lemma}[theorem]{Lemma}
\newtheorem{proposition}[theorem]{Proposition}
\newtheorem{corollary}[theorem]{Corollary}

\theoremstyle{definition}
\newtheorem{remark}[theorem]{Remark}

\newcommand{\SA}{{\mathcal{S}}_{\!A}}
\renewcommand{\Re}{\operatorname{Re}}
\renewcommand{\Im}{\operatorname{Im}}
\newcommand{\ssum}{\sideset{}{^*}\sum}
\newcommand{\Lf}{\mathcal{L}}

\allowdisplaybreaks
\begin{document}

\title[Value distribution for the derivatives of the logarithm of $L$-functions]{Value distribution for the derivatives of the logarithm of $L$-functions from the Selberg class in the half-plane of absolute convergence}

\author[T.~Nakamura]{Takashi Nakamura}
\address[T.~Nakamura]{Department of Liberal Arts, Faculty of Science and Technology, Tokyo University of Science, 2641 Yamazaki, Noda-shi, Chiba-ken, 278-8510, Japan}
\email{nakamuratakashi@rs.tus.ac.jp}
\urladdr{https://sites.google.com/site/takashinakamurazeta/}

\author[\L.~Pa\'nkowski]{{\L}ukasz Pa\'nkowski}
\address[\L.~Pa\'nkowski]{Faculty of Mathematics and Computer Science, Adam Mickiewicz University, Umultowska 87, 61-614 Pozna\'{n}, Poland, and Graduate School of Mathematics, Nagoya University, Nagoya, 464-8602, Japan}
\email{lpan@amu.edu.pl}

\maketitle

\begin{abstract}
In the present paper, we show that, for every $\delta>0$, the function $(\log {\mathcal{L}}(s))^{(m)}$, where $m\in {\mathbb{N}} \cup \{ 0\}$ and ${\mathcal{L}} (s) := \sum_{n=1}^\infty a(n) n^{-s}$ is an element of the Selberg class ${\mathcal{S}}$, takes any value infinitely often in the strip $1<\Re(s) <1+\delta$, provided $\sum_{p\leq x} |a (p)|^2 \sim \kappa\pi(x)$ for some $\kappa>0$. In particular, ${\mathcal{L}} (s)$ takes any non-zero value infinitely often in the strip $1<\Re(s)<1+\delta$, and the first derivative of ${\mathcal{L}} (s)$ has infinitely many zeros in the half-plane $\Re(s)>1$. 
\end{abstract}

\section{Introduction and statement of main results}
Let ${\mathcal{S}}_{\!A}$ consist of functions defined, for $\sigma:=\Re(s)>1$, by
\begin{equation}\label{eq:0}
{\mathcal{L}} (s) = \sum_{n=1}^\infty \frac{a(n)}{n^{s}} = 
\prod_p \exp \Biggl( \sum_{k=1}^\infty \frac{b(p^k)}{p^{ks}} \Biggr) ,
\end{equation}
where $a (n) \ll n^\varepsilon$ for any $\varepsilon >0$ and $b(p^k) \ll p^{k\theta}$ for some $\theta <1/2$. Then it is well known that both the Dirichlet series and the Euler product converge absolutely when $\sigma >1$, and $a(p)=b(p)$ for every prime $p$ (e.g. \cite[p. 112]{Steu}). Moreover, the set ${\mathcal{S}}_{\!A}$ includes the Selberg class ${\mathcal{S}}$ (for the definition we refer to \cite{Kacz} or \cite[Section 6]{Steu}), which contains a lot of $L$-functions in number theory. As mentioned in \cite[Section 2.1]{Kacz}, the Riemann zeta function $\zeta (s)$, Dirichlet $L$-functions $L(s+i\theta,\chi)$ with $\theta \in {\mathbb{R}}$ and a primitive character $\chi$, $L$-functions associated with holomorphic newforms of a congruence subgroup of ${\rm{SL}}_2 ({\mathbb{Z}})$ (after some normalization) are elements of the Selberg class. However, it should be noted that ${\mathcal{S}} \subsetneq {\mathcal{S}}_{\!A}$, since for example $\zeta (s)/\zeta (2s) \in {\mathcal{S}}_{\!A}$ but $\zeta (s)/\zeta (2s) \not \in {\mathcal{S}}$ by the fact that $\zeta (s)/\zeta (2s)$ has poles on the line $\Re (s) =1/4$. 

Many mathematicians have been studying the distribution of the logarithmic derivative of the Riemann zeta function (see eg. \cite{Go}). For instance it is known that there are some relationships between mean value of products of logarithmic derivatives of $\zeta (s)$ near the critical line, correlations of the zeros of $\zeta (s)$ and the distribution of integers representable as a product of a fixed number of prime powers (see \cite{Fa} and \cite{Go}). Moreover, it is known that the second derivative of the logarithm of the Riemann zeta function appears in the pair correlation for the zeros of $\zeta (s)$ (see for example \cite{BK}). We refer also to \cite{Sto}, where Stopple investigated zeros of $(\log \zeta (s))''$. 

In the present paper, we show the following result on value distribution of the $m$-th derivative of the logarithm of $L$-function from~$\SA$.
\begin{theorem}\label{th:dn1}
Let $m \in {\mathbb{N}} \cup \{ 0\}$, $z \in {\mathbb{C}}$ and ${\mathcal{L}} (s) := \sum_{n=1}^\infty a(n) n^{-s} \in {\mathcal{S}}_{\!A}$ satisfy 
\begin{equation}\label{eq:Selberg}
\lim_{x\to\infty}\frac{1}{\pi (x)}\sum_{p\leq x} |a(p)|^2 = \kappa
\end{equation}
for some $\kappa>0$. Then, for any $\delta>0$, we have
\begin{equation}\label{eq:1}
\# \bigl\{s: 1< \Re (s) < 1+\delta, \,\,\, \Im(s) \in [0,T] \,\,\, \mbox{and} \,\,\, 
(\log {\mathcal{L}} (s))^{(m)} = z \bigr\} \gg T
\end{equation}
for sufficiently large $T$. 
\end{theorem}
\begin{remark}
The condition \eqref{eq:Selberg} is closely related to the widely believed Selberg conjecture
\begin{equation}\label{eq:SelbergOrg}
\sum_{p\leq x}\frac{|a(p)|^2}{p} = \kappa\log\log x + O(1),\qquad(\kappa>0).
\end{equation}
However, \eqref{eq:SelbergOrg} is weaker than \eqref{eq:Selberg}, since in order to deduce \eqref{eq:Selberg} we need to assume that the error term in \eqref{eq:SelbergOrg} is $C_1+C_2/\log x + O\left((\log x)^{-2}\right)$ for $C_1,C_2\geq 0$.
\end{remark}
\begin{remark}
It is known (see for example \cite[Theorem 3.6 (vi)]{Pe}) that the assumption \eqref{eq:Selberg} implies that the abscissa of absolute convergence of $\Lf(s)\not\equiv 1$ is equal to $1$, which is also a necessary condition for \eqref{eq:1}. The main reason, why the assumption that the abscissa of absolute convergence is $1$ is not enough in our case, is the fact that we need to estimate the number of primes $p$ for which $a(p)$ is not too close to $0$. Thus, if $|a(p)|>c$ for every prime $p$ and some constant $c>0$, then \eqref{eq:1} is equivalent to the fact that the abscissa of absolute convergence is $1$.
\end{remark}

As an immediate consequence of Theorem \ref{th:dn1}, we obtain the following result.
\begin{corollary}\label{pro:1}
Let $z \in {\mathbb{C}} \setminus \{ 0 \}$ and ${\mathcal{L}} (s) \in {\mathcal{S}}_{\!A}$ satisfy \eqref{eq:Selberg}. Then, for any $\delta>0$, we have
\begin{equation}\label{eq:Bohr}
\# \bigl\{s: 1< \Re (s) < 1+\delta, \,\,\, \Im(s) \in [0,T] \,\,\, \mbox{and} \,\,\, {\mathcal{L}} (s) =z \bigr\} \gg T
\end{equation}
for sufficiently large $T$. 
\end{corollary}

The truth of \eqref{eq:Bohr} with ${\mathcal{L}} (s)= \zeta (s)$ was already proved by Bohr in \cite{Bo} (see also Remark \ref{rem:convex}). Moreover, it was conjectured in \cite[p.~188, l.~12--13]{Steu} that it holds even for all $L$-functions with the following polynomial Euler product
$$
L (s) = \sum_{n=1}^\infty \frac{a(n)}{n^{s}} = 
\prod_p \prod_{j=1}^m \biggl( 1 - \frac{\alpha_j(p)}{p^{s}} \biggr)^{-1} ,\qquad  \sigma >1,
$$
where the $\alpha_j (p)$'s are complex numbers with $|\alpha_j (p)| \le 1$ (see \cite[Section 2.2]{Steu}). One can easily show (see \cite[Lemma 2.2]{Steu}) that the coefficients $a(n)$ satisfy $a (n) \ll n^\varepsilon$ for every $\varepsilon >0$, and hence ${\mathcal{S}}_{\!A}$ contains all functions of such kind. Therefore, Corollary \ref{pro:1} shows that we have (\ref{eq:Bohr}) not only for $L(s)$ as above, but also for ${\mathcal{L}} (s) \in {\mathcal{S}}_{\!A}$. 

The above result on $c$-values is related to the  following uniqueness theorem proved by Li \cite[Theorem 1]{Li}. If two $L$-functions $\mathfrak{L}_1$ and $\mathfrak{L}_2$ (Euler product is not necessary) satisfy the same functional equation, $a_1(1)=1=a_2(1)$, and $\mathfrak{L}_1 ^{-1} (c_j) = \mathfrak{L}_2^{-1} (c_j)$ for two distinct complex numbers $c_1$ and $c_2$, then $\mathfrak{L}_1 = \mathfrak{L}_2$. It turns out (see Ki \cite[Theorem 1]{Ki}) that in the case of two functions $\mathfrak{L}_1$ and $\mathfrak{L}_2$ from the extended Selberg class $\mathcal{S}^\#$ (for definition we refer to \cite[p,~160]{Kacz} or \cite[p.~217]{Steu}), in order to prove that $\mathfrak{L}_1 = \mathfrak{L}_2$, it is sufficient to check that they have the same functional equation with positive degree, $a_1(1)=1=a_2(1)$ and $\mathfrak{L}_1 ^{-1} (c) = \mathfrak{L}_2^{-1} (c)$ for some nonzero complex number $c$. Recently, Gonek, Haan and Ki \cite{GHK} improved Ki's result by showing that the assumption that they satisfy the same functional equation is superfluous. 
Now let ${\mathcal{L}} (s) \in {\mathcal{S}}_{\!A}$ satisfy all assumptions of Corollary \ref{pro:1}. Then we see that for any $c \in {\mathbb{C}} \setminus \{0\}$ and sufficiently large $T$, we have
$$
\# {\mathcal{L}}^{-1} (c) \ge 
\# \bigl\{ s \in {\mathbb{C}} : {\mathcal{L}} (s) = c, \,\,\, \Re (s) >1, \,\,\, \Im(s) \in [0,T]\bigr\} \gg T,
$$
which means that it is not trivial to check the condition $\mathfrak{L}_1 ^{-1} (c) = \mathfrak{L}_2^{-1} (c)$ if $\mathfrak{L}_1,\mathfrak{L}_2\in \mathcal{S}_{\!A}$.

Next, since $\Lf(s)$ has no zeros in the half-plane of absolute convergence and $(\log\Lf(s))' = \Lf'(s)/\Lf(s)$, we obtain immediately the following result by using Theorem \ref{th:dn1} for $m=1$ and $z=0$.
\begin{corollary}\label{pro:d1}
Let ${\mathcal{L}} (s) \in {\mathcal{S}}_{\!A}$ satisfy \eqref{eq:Selberg}. Then for any $\delta>0$, one has
\begin{equation}\label{eq:d1}
\# \bigl\{s: 1< \Re (s) < 1+\delta, \,\,\, \Im(s) \in [0,T] \,\,\, \mbox{and} \,\,\, {\mathcal{L}}' (s) =0 \bigr\} \gg T
\end{equation}
for sufficiently large $T$. 
\end{corollary}
The above corollary generalizes the well-known result that the first derivative of the Riemann zeta function has infinitely many zeros in the region of absolute convergence $\sigma >1$ (see \cite[Theorem 11.5 (B)]{Tit}). Moreover, it should be mentioned that although there are a lot of papers on zeros of the derivatives of the Riemann zeta function (see for instance \cite{Ber}, \cite{LM}, \cite{Spei} and articles which cite them), there are few papers treating zeros of the derivatives of other zeta and $L$-functions. On result concerning this matter  is the following fact due to Yildirim \cite[Theorem 2]{Yi}. Let $\chi$ be a Dirichlet character to the modulus $q$ and $m$ be the smallest prime that does not divide $q$. Then the $k$-th derivatives of the Dirichlet $L$-function $L^{(k)} (s,\chi)$ does not vanish in the half-plane
$$
\sigma > 1 + \frac{m}{2} \Biggl( 1+\sqrt{1+\frac{4k^2}{m\log m}} \Biggr), \qquad k \in {\mathbb{N}}. 
$$

As an application of our method we show the following result concerning zeros of combinations of $L$-functions.
\begin{corollary}\label{th:1}
Let $c_1, c_2 \in {\mathbb{C}} \setminus \{ 0 \}$ and ${\mathcal{L}}_j (s) := \sum_{n=1}^\infty a_j(n) n^{-s} \in {\mathcal{S}}_{\!A}$ for $j=1,2$. Assume
\begin{equation}\label{eq:SelbTwo}
\lim_{x\to\infty}\frac{1}{\pi(x)}\sum_{p\leq x} |a_1 (p) - a_2 (p)|^2 = \kappa,\qquad (\kappa>0).
\end{equation}
Then for any $\delta>0$, it holds that
$$
\# \bigl\{s: 1 < \Re (s) < 1+\delta, \,\,\, \Im(s) \in [0,T] \,\,\, \mbox{and} \,\,\,
c_1 {\mathcal{L}}_1 (s) + c_2 {\mathcal{L}}_2 (s) =0 \bigr\} \gg T
$$
for sufficiently large $T$. 
\end{corollary}

Now we mention earlier works related to zeros of zeta functions in the half plane $\sigma >1$. Davenport and Heilbronn \cite{Daven} showed that the Hurwitz zeta function $\zeta (s, \alpha) = \sum_{n=0}^\infty (n+\alpha)^{-s}$ has infinitely many zeros in the region $\Re (s) >1$, provided $0< \alpha \ne 1/2,1$ is rational or transcendental. Later, Cassels \cite{Cassels} extended their result to algebraic irrational parameter $\alpha$. Recently, Saias and Weingartner \cite{SW} showed that  a Dirichlet series with periodic coefficients and non-vanishing in the half-plane $\sigma >1$ equals $F(s) = P(s) L(s,\chi)$, where $P(s)$ is a Dirichlet polynomial that does not vanish in $\sigma >1$. Afterwards, Booker and Thorne \cite{BT}, and very recently Righetti \cite{Reg} generalized the work of Saias and Weingartner to general $L$-functions with bounded coefficients at primes.

Nevertheless Corollary \ref{th:1} gives new examples, which cannot be treated by Saias-Weingartner approach and its known generalizations. For example Corollary \ref{th:1} implies that the Euler-Zagier double zeta function $\zeta_2 (s,s) = (\zeta^2(s)-\zeta (2s))/2$ has zeros for $\sigma >1$. Moreover, we can prove that the zeta functions associated to symmetric matrices treated by Ibukiyama and Saito in \cite[Theorem 1.2]{IbuSa1} vanish infinitely often in the region of absolute convergence. In addition, it follows that some Epstein zeta functions, for example, 
\begin{equation*}
\begin{split}
&\zeta (s; I_6) = -4 \bigl( \zeta (s) L(s-2,\chi_{-4}) - 4\zeta (s-2) L(s,\chi_{-4}) \bigr), \\
&\zeta (s; {\mathfrak{L}}_{24}) = \frac{65520}{691} \bigl( \zeta (s) \zeta (s-11) - L(s;\Delta) \bigr),
\end{split}
\end{equation*}
have infinitely many zeros for $\sigma >3$ and $\sigma >12$, respectively, since $|\tau(p)| < 2p^{11/2}$ and 
\begin{equation*}
\begin{split}
&\lim_{x\to\infty}\frac{1}{\pi(x)}\sum_{p\leq x} 
\bigl| \bigl(1+p^{-2}\bigr) \bigl(1- \chi_{-4} (p) \bigr) \bigr|^2 = 
\frac{0}{\varphi(4)} +\frac{2^2}{\varphi(4)} =2,\\
&\lim_{x\to\infty}\frac{1}{\pi(x)}\sum_{p\leq x} \bigl| p^{-11} +1 - \tau(p) p^{-11} \bigr|^2 = 1.
\end{split}
\end{equation*}
It should be noted that it is already known that $\zeta_2 (s,s)$ and $\zeta (s; {\mathfrak{L}}_{24})$ vanish in the half-plane $\sigma >1$ and $\sigma >12$ from the numerical computations \cite[Figure 1]{MaSh} and \cite[Fig.~1]{Pro}.  Furthermore, we have to remark that such zeta functions have infinitely many zeros outside of the region of absolute convergence (see \cite[Main Theorem 1]{NaPa1} and \cite[Theorem 3.1]{NaPa2}).  

In Sections 2, we prove Theorem \ref{th:dn1} and its corollaries. Some topics related to almost periodicity are discussed in Section 3. More precisely, we prove that for any $\Re (\eta) >0 $, the function $\zeta (s) \pm \zeta (s+\eta)$ has zeros when $\sigma >1$ (see Corollary \ref{cor:alzero}) but for any $\delta >0$, there exists $\theta \in {\mathbb{R}} \setminus \{ 0\}$ such that the function $\zeta (s) + \zeta (s+i\theta)$ does not vanish in the region $\sigma \ge 1+\delta$ (see Proposition \ref{pro:nozero1}). 

\section{Proofs of Theorem \ref{th:dn1} and its corollaries}

\begin{lemma}\label{lem:annulas}
Let $r_1,\ldots,r_n\in\mathbb{C}$ be such that $0<|r_1|\leq |r_2|\leq \cdots\leq |r_n|$ and $R_0 = 0$, $R_j = |r_1|+\cdots+|r_j|$. Then 
\[
\left\{\sum_{j=1}^n c_j r_j: |c_j| = 1,\ c_j\in\mathbb{C}\right\} = \{z\in\mathbb{C}: T_n\leq z\leq R_n\},
\]
where 
\[
T_n = \begin{cases}
|r_n|-R_{n-1}&\text{if $R_{n-1}\leq |r_n|$},\\
0&\text{otherwise}.
\end{cases}
\]
\end{lemma}
\begin{proof}
From \cite[Proposition 3.3]{CG} every complex number $z$ with $T_n\leq |z|\leq R_n$ can be written as 
\[
z = \sum_{j=1}^n c'_j |r_j|,\qquad |c'_j|=1.
\]
Hence, taking $c_j = c'_j |r_j|/r_j$ completes the proof.
\end{proof}

\begin{lemma}\label{lem:SN}
Let $b(p)$ be a sequence of complex numbers indexed by primes. Assume that $b(p)\ll p^{\varepsilon}$ for every $\varepsilon>0$ and   
\begin{equation*}\label{eq:PNTder}
\lim_{x\to\infty}\frac{1}{(\log x)^m\pi (x)}\sum_{p\leq x} |b(p)|^2 = \kappa
\end{equation*}
for some $\kappa>0$ and a non-negative integer $m$. Then for any $c>1$, $\eta>0$ and $\varepsilon>0$ we have
\[
\sum_{\substack{x<p\leq cx\\|b(p)|> p^{-\eta}}} 1 \gg x^{1-\varepsilon}.
\]
\end{lemma}
\begin{proof}
One can easily get that
\[
\sum_{\substack{x<p\leq cx}}|b(p)|^2 \ll x^{\varepsilon}\sum_{\substack{x<p\leq cx\\|b(p)|> p^{-\eta}}}1+x^{-2\eta}\sum_{\substack{x<p\leq cx\\|b(p)|\leq p^{-\eta}}}1
\ll x^{\varepsilon}\sum_{\substack{x<p\leq cx\\|b(p)|> p^{-\eta}}}1 + \frac{x^{1-2\eta}}{\log x}.
\]
On the other hand, we have
\begin{align*}
\sum_{\substack{x<p\leq cx}}|b(p)|^2 \gg x(\log x)^{l-1}.
\end{align*}
Hence the proof is complete.
\end{proof}

\begin{lemma}\label{lem:Cassels}
Let $L(s) = \sum_{p}\sum_{k\geq 1} {b(p^k)}{p^{-ks}}$ for $\sigma>1$ be such that $b(p^k)\ll p^{k\theta}$ for some $\theta<1/2$, $b(p)\ll p^\varepsilon$ for every $\varepsilon>0$ and  
\begin{equation}\label{eq:PNT}
\lim_{x\to\infty}\frac{1}{(\log x)^m\pi (x)}\sum_{p\leq x} |b(p)|^2 = \kappa
\end{equation}
for some $\kappa>0$  and a non-negative integer $m$.
Then, for every complex $z$ and $\delta>0$ there exist $1<\sigma<1+\delta$ and a sequence $\chi(p)$ of complex number indexed by primes such that $|\chi(p)|=1$ and 
\[
\sum_p\sum_{k\geq 1}\frac{\chi(p)^kb(p^k)}{p^{k\sigma}} = z.
\]
\end{lemma}
\begin{proof}
We follow the idea introduced by Cassels in \cite{Cassels}.

Assume that $N_1$ is a positive integer, $\varepsilon>0$ and $c_0>0$; we determine these parameters later on. Put $M_j = [c_0N_j]$ and $N_{j+1} = N_j+M_j$. We shall show that there exist $\sigma\in(1,1+\delta)$ and  a sequence $\chi(p)$ with $|\chi(p)|=1$ such that
\begin{equation}\label{eq:main}
\left|\;\sideset{}{^*}\sum_{(p,k):\,p^k\leq N_{j}}\frac{\chi(p)^k b(p^k)}{p^{k\sigma}} - z +\sum_{(p,k):\, |b(p)|\leq p^{-\varepsilon}}\frac{b(p^k)}{p^{k\sigma}}\right|\leq 10^{-2}\;\ssum_{(p,k):\,p^k>N_{j}}\frac{|b(p^k)|}{p^{k\sigma}},
\end{equation}
where ${\sum}^*$ denotes the double sum over $(p,k)$ satisfying $|b(p)|>p^{-\varepsilon}$, $p$ is prime and $k\in\mathbb{N}$. Let us note that for every $\sigma\in (1,1+\delta)$ we have
\[
\sum_{(p,k):\, |b(p)|\leq p^{-\varepsilon}}\frac{|b(p^k)|}{p^{k\sigma}}\leq \sum_{p:\, |b(p)|\leq p^{-\varepsilon}}\frac{1}{p^{1+\varepsilon}}+\sum_p\sum_{k\geq 2}\frac{|b(p^k)|}{p^{k}}=:S_0<\infty.
\]

From \eqref{eq:PNT} and \cite[Theorem 3.6 (vi)]{Pe}, the abscissa of convergence of the series $\sum_p\sum_{k\geq 1}$ ${|b(p^k)|}{p^{-k\sigma}}$ is $1$, then by Landau's theorem, this series has a pole at $\sigma=1$, which implies that
\begin{equation*}
\ssum_{(p,k)}\frac{|b(p^k)|}{p^{k\sigma}}\to\infty\qquad\text{as \ $\sigma\to 1^+$}.
\end{equation*}

Therefore, we can find $\sigma\in (1,1+\delta)$ such that
\[
\ssum_{(p,k):\,p^k\leq N_1}\frac{|b(p^k)|}{p^{k\sigma}} + |z| + S_0
\leq 10^{-2}\;\ssum_{(p,k):\,p^k>N_{1}}\frac{|b(p^k)|}{p^{k\sigma}},
\]
and hence \eqref{eq:main} holds for $j=1$ and arbitrary $\chi(p)$'s with $p\leq N_1$.

Now, let us assume that complex numbers $\chi(p)$ are chosen for all $p\leq N_j$. We shall find $\chi(p)$ with $N_j<p\leq N_{j+1}$ and $|b(p)|>p^{-\varepsilon}$ such that \eqref{eq:main} holds with $j+1$ instead of $j$.

Let $\mathfrak{A}$ denote the set of pairs $(p,1)$ satisfying $p\in (N_j,N_{j+1}]$ is a prime number and $|b(p)|>p^{-\varepsilon}$. Moreover, define 
\[
\mathfrak{B}=\{(p,k): p^k\in (N_j,N_{j+1}],\ \text{$p$ is prime},\ k\geq 2,\ |b(p)|>p^{-\varepsilon}\}.\]

Note that the $\chi(p)^k$'s are already defined for $(p,k)\in\mathfrak{B}$, since for suitable $N_1$ and $c_0$ we have $p \leq \sqrt{N_{j+1}}<N_j$ if $(p,k)\in\mathfrak{B}$.

Using Lemma \ref{lem:SN} gives that
\[
|\mathfrak{A}|\gg N^{1-\varepsilon}_j
\]
and since $k\geq 2$ for every $(p,k)\in\mathfrak{B}$ we have
\[
|\mathfrak{B}|\ll N^{\frac{1}{2}}_j.
\]

Moreover, note that for every $p_1,p_2$ satisfying $(p_1,1),(p_2,1)\in\mathfrak{A}$, by the Ramanujan conjecture, we have
\[
\left|\frac{b(p_1)}{b(p_2)}\right|\ll N_j^{2\varepsilon}\qquad \text{and}\qquad\left(\frac{p_2}{p_1}\right)^\sigma\leq \left(\frac{N_{j+1}}{N_j}\right)^\sigma\leq (c_0+1)^{1+\delta},
\]
so
\[
\frac{|b(p_2)|}{p_2^\sigma}\gg N_j^{-2\varepsilon}\frac{|b(p_1)|}{p_1^\sigma}.
\]

Hence, using Lemma \ref{lem:annulas} with the sequence ${b(p)}{p^{-\sigma}}$, where $(p,1)\in\mathfrak{A}$, we obtain that
\[
\ssum_{(p,1)\in\mathfrak{A}}\frac{b(p)\chi(p)}{p^\sigma},\qquad |\chi(p)|=1,
\]
takes all values $z_0$ with $|z_0|\leq \ssum_{(p,1)\in\mathfrak{A}}{|b(p)|}{p^{-\sigma}}=:S_3$, since for sufficiently large $N_1$ and arbitrary $p_0$ satisfying $(p_0,1)\in\mathfrak{A}$, we have
\[
\ssum_{(p_0,1)\ne (p,1)\in\mathfrak{A}}\frac{|b(p)|}{p^\sigma}\gg N_j^{1-3\varepsilon}\frac{|b(p_0)|}{p_0^\sigma}>\frac{|b(p_0)|}{p_0^\sigma}.
\]
Hence the inner radius $T_{|\mathfrak{A}|}$ in Lemma \ref{lem:annulas} is $0$.

Write
\[
\Lambda:=\ssum_{(p,k):\,p^k\leq N_{j}}\frac{\chi(p)^k b(p^k)}{p^{k\sigma}} - z +\sum_{(p,k):\, |b(p)|\leq p^{-\varepsilon}}\frac{b(p^k)}{p^{k\sigma}}+\ssum_{(p,k)\in\mathfrak{B}}\frac{\chi(p)^k b(p^k)}{p^{k\sigma}}
\]
and put
\[
z_0 = \begin{cases}
-\Lambda&\text{ if $0<|\Lambda|\leq S_3$},\\
-S_3\Lambda/|\Lambda|&\text{ if $|\Lambda|> S_3$},\\
0&\text{ if $\Lambda=0$.}\\
\end{cases}
\]
Then, from Lemma \ref{lem:annulas} we can choose $\chi(p)$ for $(p,1)\in\mathfrak{A}$ such that
\begin{equation*}
\begin{split}
&\left|\;\ssum_{(p,k):\,p^k\leq N_{j}+M_j}\frac{\chi(p)^k b(p^k)}{p^{k\sigma}} - z +\sum_{(p,k):\, |b(p)|\leq p^{-\varepsilon}}\frac{b(p^k)}{p^{k\sigma}}\right|\\
&\qquad \qquad  =\left|\Lambda + \ssum_{(p,1)\in\mathfrak{A}}\frac{b(p)\chi(p)}{p^\sigma}\right|
\leq \max(0,S_1+S_2-S_3),
\end{split}
\end{equation*}
where
\[
S_1:=\left|\;\ssum_{(p,k):\,p^k\leq N_{j}}\frac{\chi(p)^k b(p^k)}{p^{k\sigma}} - z +\sum_{(p,k):\,|b(p)|\leq p^{-\varepsilon}}\frac{b(p^k)}{p^{k\sigma}}\right|
\]
and
\[
S_2:=\ssum_{(p,k)\in\mathfrak{B}}\frac{|b(p^k)|}{p^{k\sigma}},
\]
so $|\Lambda|\leq S_1+S_2$.

Now, let us notice that
\[
\frac{S_3}{S_2}\geq \frac{N_j^{\sigma-\theta}}{N_{j+1}^{\sigma+\varepsilon}}\frac{|\mathfrak{A}|}{|\mathfrak{B}|}\gg N_j^{1/2-\theta-2\varepsilon}\geq \frac{101}{99}
\]
for sufficiently small $\varepsilon>0$ and sufficiently large $N_1$. Hence
\[
S_2-S_3\leq -10^{-2}(S_2+S_3).
\]

Moreover, from \eqref{eq:main} we have
\[
S_1\leq 10^{-2}(S_2+S_3+S_4),
\]
where
\[
S_4:=\ssum_{(p,k):\,p^k>N_{j+1}}\frac{|b(p^k)|}{p^{k\sigma}}.
\]
Thus $S_1+S_2-S_3<10^{-2}S_4$ and, by induction, \eqref{eq:main} holds for all $j\in\mathbb{N}$. So letting $N_j\to\infty$ completes the proof.
\end{proof}

The classical Kronecker approximation theorem (see for example \cite[Lemma 1.8]{Steu}) plays a crucial role in the proof of the following lemma.
\begin{lemma}\label{lem:4}
Assume that $L(s)$ satisfies the hypothesis of Lemma \ref{lem:Cassels}.
Then, for every $z$ and $\delta>0$, the set of real $\tau$ satisfying
\[
L(s+i\tau)=z\qquad\text{for some $1<\Re(s)<1+\delta$,}
\]
has a positive lower density. 
\end{lemma}

\begin{proof}
By Lemma \ref{lem:Cassels}, we choose $\sigma\in (1,1+\delta)$ and a sequence $\chi(p)$ with $|\chi(p)|=1$ such that
\[
\sum_p\sum_{k\geq 1}\frac{\chi(p)^kb(p^k)}{p^{k\sigma}} = z.
\]
Next, since $F(s)=\sum_p\sum_{k\geq 1}{\chi(p)^kb(p^k)}{p^{-ks}}$ is analytic in the half-plane $\Re(s)>1$, we can find $r$ with $0<r<\sigma-1$ such that $F(s)-z\ne 0$ if $|s-\sigma|=r$. Then we put $\varepsilon:=\min_{|s-z|=r}|F(s)-z|$.

Since the series $\sum_{p} \sum_{k=1}^\infty |b(p^k)| p^{-k(\sigma-r)}$ converges absolutely, we can take a positive integer $M$ such that
\begin{equation*}\label{eq:M0}
\sum_{p\le M} \sum_{k>M}^\infty \frac{|b(p^k)|}{p^{k(\sigma-r)}} + 
\sum_{p>M} \sum_{k=1}^\infty \frac{|b(p^k)|}{p^{k(\sigma-r)}} < \frac{\varepsilon}{4}.
\end{equation*}

Moreover, if we assume that
\begin{equation}\label{eq:dioph}
\max_{p\leq M}\left|p^{-i\tau} - \chi(p)\right|<\varepsilon_1
\end{equation}
for $\varepsilon_1>0$, then
\begin{equation*}
\begin{split}
\bigl|p^{-ik\tau} - \chi(p)^k\bigr| &= |p^{-i\tau} - \chi(p)| 
\bigl|p^{-i(k-1)\tau} + p^{-i(k-2)\tau} \chi(p) + \cdots + p^{-i\tau} \chi(p)^{k-2} + \chi(p)^{k-1}\bigr| \\
& < k \varepsilon_1 \le  M \varepsilon_1, \qquad 1 \le k \le M.
\end{split}
\end{equation*}

Therefore, for sufficiently small $\varepsilon_1$ and $s$ satisfying $|s-\sigma|= r$, we obtain
$$
\Biggl| \sum_{p \le M} \sum_{k=1}^M \frac{b(p^k)}{p^{k(s+i\tau)}} -
\sum_{p \le M} \sum_{k=1}^M \frac{b(p^k) \chi(p)^k}{p^{ks}}  \Biggr| <
M \varepsilon_1 \sum_{p \le M} \sum_{k=1}^M \frac{|b(p^k)|}{p^{k(\sigma-r)}}<\frac{\varepsilon}{2},
$$
and
\[
\left|L(s+i\tau)-z-(F(s)-z)\right| = \left|L(s+i\tau)-F(s)\right| <\varepsilon \leq |F(s)-z|,
\]
provided \eqref{eq:dioph} holds. 

Thus, by Rouch\'e's theorem (see for example \cite[Theorem 8.1]{Steu}), for every $\tau$ satisfying \eqref{eq:dioph} there is a complex number $s$ with $|s-\sigma|\leq r$ such that  $L(s+i\tau) = z$. But, by the classical Kronecker approximation theorem,  the set of $\tau$ satisfying \eqref{eq:dioph} has a positive density, so the number of solutions of the equation $L(s+i\tau)=z$ with $1<\Re(s)<1+\delta$ and $\tau\in[0,T]$ is $\gg T$ for sufficiently large $T>0$.
\end{proof}

Now we are in a position to prove Theorem \ref{th:dn1}. 
\begin{proof}[Proof of Theorem \ref{th:dn1}]
Obviously, the case $m=0$ follows immediately from Lemma \ref{lem:4}, since $a(p)=b(p)$ for every prime $p$. Thus it suffices to show that for every $m\geq 1$ the function $(\log \Lf(s))^{(m)}$ satisfies the assumption of Lemma \ref{lem:4}.

Note that
$$
(-1)^m\bigl( \log {\mathcal{L}} (s)\bigr)^{(m)}  =
 \sum_{p} \sum_{k=1}^\infty \frac{b(p^k) (k\log p)^m}{p^{ks}} , \qquad \sigma >1,
$$
$b(p)(\log p)^m= a(p)(\log p)^m\ll p^{\varepsilon}$ for every $\varepsilon>0$, and $b(p^k) (k\log p)^m \ll p^{k\theta_1}$ for some $\theta_1$ with $\theta<\theta_1<1/2$ by the assumption $b(p^k) \ll p^{k\theta}$ for some $\theta <1/2$. Moreover, by partial summation and \eqref{eq:Selberg}, we get
\begin{align*}
\sum_{p\leq x}|b(p)|^2(\log p)^{2m} = \sum_{p\leq x}|a(p)|^2(\log p)^{2m}=\kappa(\log x)^{2m}\pi(x)(1+o(1)),
\end{align*}
which completes the proof.
\end{proof}

\begin{remark}\label{rem:convex}
In Bohr's proof of Corollary \ref{pro:1} for ${\mathcal{L}} (s) = \zeta (s)$, the convexity of 
$$
- \log \bigl( 1-p^{-s}\bigr) = \sum_{k=1}^\infty \frac{1}{kp^{ks}}
$$
plays a crucial role (see also \cite[Theorem 1.3]{Steu} and \cite[Theorem 11.6 (B)]{Tit}). However, we prove Corollary \ref{pro:1} without using the convexity since the closed curve described by $\sum_{k=1}^\infty b(p^k) p^{-ks}$ is not always convex when $t$ runs through the whole ${\mathbb{R}}$ (see also \cite{Map}).
\end{remark}

\begin{proof}[Proof of Corollary \ref{th:1}]
Put $L(s)=\log L_1(s) - \log L_2(s)$. Then $L(s) = \sum_p\sum_{k\geq 1}(b_1(p^k)-b_2(p^k))p^{-ks}$, where the $b_j(p^k)$'s denote the coefficients in the Dirichlet series expansion of $\log L_j(s)$. Obviously $b_1(p^k)-b_2(p^k)\ll p^{k\theta}$ for some $\theta<1/2$, and $b_1(p)-b_2(p)\ll p^\varepsilon$ for every $\varepsilon>0$. Thus, by \eqref{eq:SelbTwo}, we can apply Lemma \ref{lem:4} to obtain that the set of real $\tau$ satisfying 
\[
L(s+i\tau) = \log (L_1(s+i\tau)/L_2(s+i\tau))= \log(-c_2/c_1)
\]
has a positive lower density. Therefore, the proof is complete.
\end{proof}
\begin{remark}\label{rem:1}
Note that
\begin{equation}\label{eq:DiffTwo}
\sum_{p\leq x} |a_1 (p) - a_2 (p)|^2 = \sum_{p\leq x} |a_1 (p)|^2 + \sum_{p\leq x} |a_2 (p)|^2 -2\Re\sum_{p\leq x} a_1 (p) \overline{a_2 (p)}.
\end{equation}
Therefore, if the abscissa of absolute convergence for both $L$-functions $\Lf_1$ and $\Lf_2$ is $1$, then the assumption \eqref{eq:SelbTwo} in Corollary \ref{th:1} can be replaced by Selberg's orthonormality conjecture in the following stronger form
\[
\forall_{j=1,2}\lim_{x\to\infty}\frac{1}{\pi(x)}\sum_{p\leq x} |a_j(p)|^2=\kappa_j,\qquad \lim_{x\to\infty}\frac{1}{\pi(x)}\sum_{p\leq x} a_1(p)\overline{a_2(p)} = 0,
\]
for some $\kappa_1,\kappa_2>0$.

On the other hand, if the abscissa of absolute convergence of one of them, say $\Lf_2$, is less than $1$, then we get
\[
\sum_{p\leq x}|a_2(p)|^2\leq \sqrt{\sum_{p\leq x}\frac{|a_2(p)|}{p^{\sigma_0}}}\sqrt{\sum_{p\leq x}|a_2(p)|^3p^{\sigma_0}} \ll x^{1/2+\sigma_0/2+\varepsilon}
\]
for some $\sigma_0<1$ and every $\varepsilon>0$.
Moreover, by Cauchy-Schwarz inequality, we have
\[
\Re\sum_{p\leq x} a_1 (p) \overline{a_2 (p)}\leq \sqrt{\sum_{p\leq x}|a_1(p)|^2}\sqrt{\sum_{p\leq x}|a_2(p)|^2}\ll x^{3/4+\sigma_0/4+\varepsilon}
\]
for every $\varepsilon>0$.

Therefore, by \eqref{eq:DiffTwo}, we obtain
\begin{align*}
\sum_{p\leq x} |a_1 (p) - a_2 (p)|^2 = \sum_{p\leq x}|a_1(p)|^2 + O(x^{3/4+\sigma_0/4+\varepsilon}),
\end{align*}
and assuming \eqref{eq:Selberg} for $\Lf_1$ implies Corollary \ref{th:1}.
\end{remark}

\section{Almost periodicity and Corollary \ref{th:1}}
We follow the notion of almost periodicity in \cite[Section 9.5]{Steu}. In 1922, Bohr \cite{Bohr} proved that every Dirichlet series $f(s)$, having a finite abscissa of absolute convergence $\sigma_a$ is almost periodic in the half-plane $\sigma > \sigma_a$. Namely, for any given $\delta >0$ and $\varepsilon >0$, there exists a length $l:= l(f,\delta,\varepsilon)$ such that every interval of length $l$ contains a number $\tau$ for which
$$
\bigl| f(\sigma + {\rm{i}}t + {\rm{i}}\tau) - f(\sigma + {\rm{i}}t) \bigr| < \varepsilon
$$
holds for any $\sigma \ge \sigma_a+\delta$ and for all $t \in {\mathbb{R}}$. From the Dirichlet series expression, the $L$-function ${\mathcal{L}} (s) \in {\mathcal{S}}_{\!A}$ is almost periodic when $\sigma >1$. By using Corollary \ref{th:1}, we have the following corollary as a kind of analogue of almost periodicity.
\begin{corollary}\label{cor:alzero}
Let ${\mathcal{L}} (s) := \sum_{n=1}^\infty a(n) n^{-s} \in {\mathcal{S}}_{\!A}$ satisfy \eqref{eq:Selberg}. Suppose $c_1, c_2 \in {\mathbb{C}} \setminus \{ 0 \}$ and $\Re(\eta) >0$. Then one has 
$$
\# \bigl\{s: \Re (s) >1, \,\,\, \Im(s) \in [0,T] \,\,\, \mbox{and} \,\,\,
c_1 {\mathcal{L}} (s) + c_2 {\mathcal{L}} (s+\eta) =0 \bigr\} \gg T
$$
for sufficiently large $T$. 
\end{corollary}
\begin{proof}
The corollary follows from Remark \ref{rem:1}, since the abscissa of absolute convergence of $\Lf(s+\eta)$ is smaller than $1$.
\end{proof}

On the contrary, we have the following proposition when $\Re(\eta) =0$.
\begin{proposition}\label{pro:nozero1}
Let ${\mathcal{L}} (s) \in {\mathcal{S}}_{\!A}$. Then for any $\delta >0$, there exists $\theta \in {\mathbb{R}} \setminus \{0 \}$ such that the function 
$$
{\mathcal{L}} (s) + {\mathcal{L}} (s+i\theta)
$$ 
does not vanish in the region $\sigma \ge 1+\delta$. 
\end{proposition}
\begin{proof}
For any $\varepsilon >0$, we can find $\theta \in {\mathbb{R}} \setminus \{0 \}$ which satisfies
$$
|{\mathcal{L}} (s) - {\mathcal{L}} (s+i\theta)| < \varepsilon, \qquad \Re (s) \ge 1+\delta
$$ 
from almost periodicity of ${\mathcal{L}} (s) \in {\mathcal{S}}_{\!A}$. Hence we have
\begin{equation*}
\begin{split}
\bigl|{\mathcal{L}} (s) + {\mathcal{L}} (s+i\theta) \bigr| &= 
\bigl|2{\mathcal{L}} (s) + {\mathcal{L}} (s+i\theta) - {\mathcal{L}} (s) \bigr| \ge
\bigl|2{\mathcal{L}} (s) \bigr| - \bigl|{\mathcal{L}} (s) - {\mathcal{L}} (s+i\theta) \bigr| \\ &>
2 \prod_p \exp \Biggl( - \sum_{k=1}^\infty \frac{|b(p^k)|}{p^{k(1+\delta)}} \Biggr) - \varepsilon ,
\qquad \Re (s) \ge 1+\delta.
\end{split}
\end{equation*}
From the assumption for ${\mathcal{L}} (s) \in {\mathcal{S}}_{\!A}$, the sum $\sum_p \sum_{k=1}^\infty |b(p^k)| p^{-k(1+\delta)}$ converges absolutely when $\delta >0$. Hence, by taking a suitable $\varepsilon >0$, we have
$$
|{\mathcal{L}} (s) + {\mathcal{L}} (s+i\theta)| > 0, \qquad \Re (s) \ge 1+\delta,
$$
which implies Proposition \ref{pro:nozero1}. 
\end{proof}

\begin{remark}
Proposition \ref{pro:nozero1} should be compared with the following fact. Let $\theta \in {\mathbb{R}} \setminus \{0 \}$ and $c_1, c_2 \in {\mathbb{C}} \setminus \{ 0 \}$. Then the function 
$$
c_1 \zeta (s) + c_2 \zeta (s+i\theta)
$$
vanishes in the strip $1/2 < \sigma <1$. This is an easy consequence of \cite[Theorem 10.7]{Steu}. 

Hence, for any $\delta >0$, there exists $\theta \in {\mathbb{R}} \setminus \{0 \}$ such that the function 
$$
\zeta (s) + \zeta (s+i\theta)
$$ 
does not vanish in the half-plane $\sigma \ge 1+\delta$, but has infinitely many zeros in the vertical strip $1/2 < \sigma <1$. 
\end{remark}

\subsection*{Acknowledgments}
The first author was partially supported by JSPS grant 24740029. 

The second author was partially supported by (JSPS) KAKENHI grant no. 26004317 and the grant no. 2013/11/B/ST1/02799 from the National Science Centre.

The authors would like to thank the referee for useful comments and suggestions that helped them to improve the original manuscript. 


\end{document}